\documentclass[12pt]{amsart}
\usepackage{amsfonts, amstext}
\usepackage[usenames,dvipsnames]{xcolor}	
\usepackage{amsmath, amsthm, amssymb, amsfonts}						
\usepackage{fullpage}
\usepackage{enumitem}
\usepackage{mathrsfs}
\usepackage{tikz-cd}
\usepackage{hhline}
\usepackage{textcomp}
\usepackage{mathbbol}

\usepackage{hyperref}

\newcommand{\mscr}{\mathscr}
\newcommand*{\sHom}{\mscr{H}\kern -.5pt om}
\newcommand*{\sExt}{\mscr{E}\kern -.5pt xt}

\newcommand{\lto}{\leftarrow}

\newcommand{\PP}{\mathbb{P}}

\newcommand{\QQ}{\mathbb{Q}}

\newcommand{\ZZ}{\mathbb{Z}}

\newcommand{\VV}{\mathbf{V}}

\newcommand{\ii}{\mathbf{i}}
\newcommand{\ee}{\mathbf{e}}

\DeclareMathOperator{\Tor}{Tor}

\DeclareMathOperator{\reg}{reg}

\DeclareMathOperator{\HN}{HN}
\DeclareMathOperator{\mult}{mult}

\newtheorem{theorem}{Theorem}[section]

\newtheorem{corollary}[theorem]{Corollary}
\newtheorem{lemma}[theorem]{Lemma}

\theoremstyle{definition}
\newtheorem{example}[theorem]{Example}

\theoremstyle{remark}
\newtheorem{remark}[theorem]{Remark}

\title{Asymptotic Syzygies of Secant Varieties of Curves}
\author{Gregory Taylor}
\address{Department of Mathematics, Statistics and Computer Science, University of Illinois at Chicago, Chicago, IL, 60607}
\email{gtaylo9@uic.edu}

\begin{document}

\maketitle

\begin{abstract}
	We prove that the minimal free resolution of the secant variety of a curve is asymptotically pure. As a corollary, we show that the Betti numbers of converge to a normal distribution.
\end{abstract}

\section{Introduction}\label{sec:Intro}


Let $C$ be a smooth curve of genus $g\geq 1$ over an algebraically closed field $\mathbb{k}$ of characteristic zero. Consider the embedding of $C$ in $\PP H^0(C,L_d) = \PP^{r_d}$  by a complete linear system of degree $d$. Since Green's fundamental work on Koszul cohomology \cite{Green84}, it is well understood that the geometry of $C$ is reflected in the minimal free resolution of the homogeneous coordinate ring of $C$ as an $S = \mathbb{k}[x_0,\dots,x_{r_d}]$-module. One result along these lines concerns Boij-S\"{o}derberg decompositions. Boij-S\"{o}derberg theory \cite{EisenbudSchreyer} shows that every Betti table is a nonnegative linear combination of pure diagrams, the simplest diagrams of $S$-modules. In \cite{Erman}, Erman shows that as $d$ grows, the Boij-S\"{o}derberg decomposition is dominated by a single pure diagram which only depends on the genus of $C$.

In \cite{SidmanVermeire2}, Sidman and Vermeire proved that Green's celebrated ``$2g+p+1$" Theorem \cite{Green84} extends to the secant variety of lines in a natural way, and they conjectured an extension to the secant variety of $k$-planes, which we denote by $\Sigma_k(C,L_d)$. Recently Ein, Niu, and Park proved these conjectures \cite{EinNiuPark} (see Section \ref{sec:Background} for precise statements). Not much is known about syzygies of higher dimensional varieties, due in no small part to a dearth of examples \cite{EinLazAsymptotic}. Secant varieties of curves provide valuable examples of higher dimensional varieties whose syzygies we can understand.

Our main result is an analogue of Erman's theorem for secant varieties of curves. That is, the Boij-S\"{o}derberg decomposition of the Betti table of the homogeneous coordinate ring of $\Sigma_k(C,L_d)$ is increasingly dominated by a single pure diagram as $d$ grows. Moreover, this diagram only depends on $k$ and the genus of $C$. See Section \ref{sec:Background} for a brief introduction to Boij-S\"{o}derberg decompositions and  Section 3 for an explanation of the (nonstandard) notation for the pure diagrams $\pi_k(\ii;d)$.

\begin{theorem}[Asymptotic Purity]\label{thm:mainIntro}
	The normalized Betti table $\overline{\beta}(\Sigma_{k}(C,L_d))$ converges to $\pi_k((g,\dots,g);d)$ as $d \to \infty$. That is, given a Boij-S\"{o}derberg decomposition $\overline{\beta}(\Sigma_{k}(C,L_d)) = \sum_{\ii} c_{\ii;d} \pi_k(\ii;d)$ for each $d$, we have
	\[
		\lim_{d \to \infty} c_{\ii;d} = \begin{cases}
			1 & \ii = (g,\dots,g) \\
			0 & \text{otherwise}
		\end{cases}.
	\]
\end{theorem}

As pure diagrams are the building blocks of all Betti tables, Theorem 1.1 states that as the degree of the embedding grows, the Betti diagram of $\Sigma_k(C,L)$ becomes as simple as possible. Being a direct generalization of Erman's theorem \cite{Erman} for curves, this result further suggests that we should expect theorems about syzygies of curves to admit generalizations to secant varieties. 

As an application of Theorem \ref{thm:mainIntro}, we determine the asymptotic distribution of the Betti numbers of $\Sigma_k(C,L_d)$. In particular, they are essentially normally distributed, as one might expect given the work of Ein, Erman, Lazarsfeld \cite{EinErmanLaz}. In the statement below $\kappa_{p,q}(\Sigma_k(C,L_d))$ denotes the dimension of the Koszul cohomology group of degree $p$ and weight $q$ (see Section \ref{sec:Background} for more details).

\begin{theorem}[Normal Distribution]\label{thm:normalDistributionIntro}
	Let $L_d$ be a line bundle of degree $d$ on $C$. Fix a sequence $\{p_d\}$ of integers such that $p_d \to \frac{r_d}{2} + \frac{a\sqrt{r_d}}{2}$ for some real number $a$. Then as $d \to \infty$
	\[
		F_{g,k}(d)\kappa_{p_d,k+1}(\Sigma_k(C,L_d)) \to e^{-a^2}
	\]
	where $F_{g,k}(d) = \frac{(k+1)!}{2^{r_d - 2k} r_d^{k}} \cdot \sqrt{\frac{2\pi}{r_d}}$.
\end{theorem}

Note that the function $F_{g,k}$ controlling the rate of convergence depends only on $k$ and $g$. For $k = 0$ (i.e. for high degree embeddings of curves), Theorem \ref{thm:normalDistributionIntro} was observed by Ein, Erman, and Lazarsfeld in \cite[Proposition A]{EinErmanLaz} using an explicit formula for the Betti numbers. They conjectured that a similar statement should hold for higher dimensional varieties. As heuristic evidence for the conjecture, they demonstrate that the rows of a sequence of ``randomly chosen" Betti tables become normally distributed as $d \to \infty$. Some other special cases of normally distributed Betti numbers appear in \cite{BEGY20, Bruce, ErmanYang}.

To prove Theorem \ref{thm:mainIntro}, we use Erman's technique from \cite{Erman}, leveraging the powerful results from \cite{EinNiuPark}. In particular, we study the limiting behavior of the Hilbert numerator of the homogeneous coordinate ring of $\Sigma_k(C,L_d)$. For Theorem \ref{thm:normalDistributionIntro}, we leverage the fact that the Betti diagram is asymptotically pure. Since the Betti numbers of the pure diagrams have similar asymptotic behavior, we need only study the asymptotics of the pure diagram which dominates the decomposition, for which there are explicit formulas.

\subsection*{Acknowledgements}
I would like to thank my advisor, Kevin Tucker, for his support and encouragement. I also thank Daniel Erman for his insightful suggestions. Thanks to the anonymous referee for a close reading which improved the clarity of the text. The author was supported by grant number 1246844 from the National Science Foundation.

\section{Background}\label{sec:Background}


\subsection{Koszul Cohomology and Secant Varieties} 
Let $X \subseteq \PP H^0(X,L) =\PP^r$ be a projective variety embedded by a line bundle $L$, and let $S = \mathbb{k}[x_0,\dots,x_r]$. Let $R(X,L)$ be the homogeneous coordinate ring of $X$, which is a graded $S$-module. Hence, $R(X,L)$ admits a minimal free resolution
\[
	0 \lto R(X,L) \lto F_0 \lto F_1 \lto \cdots \lto F_n \lto 0
\]
where each $F_p$ is a free module. We write $F_p = \bigoplus_{q} K_{p,q}(X,L) \otimes_{k} S(-p-q)$ where $K_{p,q}(X,L)$ is a $\mathbb{k}$-vector space whose dimension counts the independent \emph{syzygies of degree $p$ and weight $q$}. The spaces $K_{p,q}(X,L)$ are the \emph{Koszul cohomology groups} \cite{AproduNagel}. Note that $K_{p,q}(X,L) \cong \Tor_p^S(R(X,L),\mathbb{k})_{p+q}$. The numbers $\kappa_{p,q}(X,L) = \dim K_{p,q}(X,L)$ are called \emph{Betti numbers}\footnote{Some authors use the convention $\beta_{i,j}(X,L) = \kappa_{i,j-i}(X,L)$}. We suppress the linear system in the notation when no confusion can arise. We express the multiplicity of the graded pieces in the \emph{Betti table}
\[
	\begin{array}{c|cccc}
		 & 0 & 1 & 2 & \cdots  \\ \hline
		0 & \kappa_{0,0} & \kappa_{1,0} & \kappa_{2,0} & \cdots \\
		1 & \kappa_{0,1} & \kappa_{1,1} & \kappa_{2,1} & \cdots  \\
		2 & \kappa_{0,2} & \kappa_{1,2} & \kappa_{2,2} & \cdots  \\
		\vdots & & \vdots & & \ddots 
	\end{array}
\]
which we denote by $\beta(X,L)$. The number of non-zero rows in the Betti diagram is the \emph{Castelnuovo-Mumford regularity} of $X$ in $\PP^r$. Note that $X \subseteq \PP^r$ is projectively normal if and only if $F_0 = S$ if and only if 
\[
	\kappa_{0,q}(X,L) = \begin{cases}
		1 & q = 0 \\
		0 & q >0
	\end{cases}.
\]
In this case, $\kappa_{1,q}(X,L)$ is equal to the number of minimal generators of the ideal $I_{X/\PP^r}$ of degree $q+1$. Following \cite{EinNiuPark}, we say that $X \subseteq \PP^r$ satisfies condition $N_{k+2,p}$ if $K_{i,j}(X,L) = 0$ for $i \leq p$ and $j \geq k+2$. When $X$ is projectively normal and $k = 0$, this coincides with Green's condition $N_p$ \cite{Green84} which occurs when the homogeneous ideal $I_{X/\PP^r}$ is generated by quadrics and the resolution is linear up to the $p$th step.

From now on, $C$ denotes a smooth curve of genus $g \geq 1$, and $L_d$ is a line bundle of degree $d$ on $C$. In his landmark paper on Koszul cohomology \cite{Green84}, Green showed that when $d \geq 2g+p+1$, the embedding $C \subseteq \PP H^0(C,L)$ satisfies $N_{2,p}$. Recently, Ein, Niu, and Park showed that Green's theorem is a special case of a more general statement about the syzygies of the secant varieties to $C$. The \emph{secant variety of $k$-planes} to $C$ in $\PP^r$ is the closure of the points contained in the $k$-planes generated by $k+1$ points of $C$.  The Ein-Niu-Park paper completes earlier work of Vermeire \cite{Ver1,Ver2,Ver3,Ver4} and Sidman-Vermeire \cite{SidmanVermeire,SidmanVermeire2} on the case $k=1$. We summarize the results from \cite{EinNiuPark} that we will use in the sequel.

\begin{theorem}[{\cite{EinNiuPark}}]\label{thm:ENP}
	Let $C$ be a curve of genus $g \geq 1$, and let $L_d$ be a line bundle of degree $d = 2g + 2k + p + 1$. Let $\Sigma_k(C,L_d)$ be the secant variety of $k$-planes to $C \subseteq \PP H^0(C,L_d) = \PP^{r_d}$. Then
	\begin{enumerate}
		\item \label{thm:ENP-part:arithCM} $\Sigma_k(C,L_d)$ is arithmetically Cohen-Macaulay (i.e. $R(\Sigma_k(C,L_d))$ has projective dimension $r_d-2k-1$),
		\item \label{thm:ENP-part:regularity} If $g \geq 1$, the ideal sheaf of $\Sigma_k(C,L_d)$ has Castelnuovo-Mumford regularity $2k+3$. Otherwise, it has regularity $k+2$,
		\item \label{thm:ENP-part:Nk+2,p} $\Sigma_k(C,L_d)$ is projectively normal and satisfies $N_{k+2,p}$,
		\item \label{thm:ENP-part:FinalBettiNumber} The bottom-right Betti number is given by
		\[
			\kappa_{r_d-2k-1,2k+2}(\Sigma_k(C,L_d)) = {g+k \choose k+1}.
		\]
		\item \label{thm:ENP-part:degree} The degree of $\Sigma_k(C,L_d)$ is given by
		\[
			\deg(\Sigma_k(C,L_d)) = \sum_{i=0}^{\min\{k+1,g\}} {d-g-k-i \choose k+1-i} {g \choose i}.
		\]
	\end{enumerate}
\end{theorem}
\begin{proof}
	Parts (\ref{thm:ENP-part:arithCM})-(\ref{thm:ENP-part:FinalBettiNumber}) are Theorem 5.8, Corollary 5.9(3), Theorem 5.2(4), and Corollary 5.9(1), respectively in \cite{EinNiuPark}. Part (\ref{thm:ENP-part:degree}) is proven in  \cite[Theorem 1]{Soule}.
\end{proof}

\begin{remark}\label{rmk:whatAboutP1}
	We ignore the case $C = \PP^1$ because $\Sigma_k(\PP^1,L_d)$ is known to have a pure Betti diagram. Indeed, the resolution is an Eagon-Northcott complex associated to the $(k+1)\times(k+1)$ minors of a 1-generic matrix. This is proved in \cite{Eis88}, where the result is attributed to unpublished work of R. K. Wakerling. 
\end{remark}

\begin{remark}\label{rmk:previousWork}
	Parts (\ref{thm:ENP-part:arithCM})-(\ref{thm:ENP-part:FinalBettiNumber}) were conjectured in \cite{SidmanVermeire2}, where they proved the conjectures for $k = 1$. For elliptic curves, they were proven for all $k$ in \cite{GH04}.
\end{remark}

It is well known that $\Sigma_k(C,L_d)$ does not lie on any hypersurface of degree less than $k+2$ \cite{CJ01}. Thus for $\Sigma_k(C,L_d)$, the condition $N_{k+2,p}$ means that its ideal is generated in degree $k+2$ and the resolution is linear to the $p$th step. This fact along with parts (\ref{thm:ENP-part:arithCM}) and (\ref{thm:ENP-part:regularity}) of Theorem \ref{thm:ENP} imply that $\beta(\Sigma_k(C,L_d))$ is of the form
\begin{equation}\label{eq:secantBettiTable}
	\begin{array}{c|cccccccc}
		& 0 & 1 & 2  & \cdots & r_d-g-2k-1 & r_d-g-2k & \cdots & r_d-2k-1 \\ \hline
		0 & 1  & - & - & \cdots  & - & - & \cdots   & - \\
		1 & -  & - & - & \cdots & - & - & \cdots & - \\
		\vdots & & \vdots  &  &  & \vdots & \vdots &  & \vdots  \\
		k & - & - & - & \cdots & - & - & \cdots & - \\
		k+1 & - & * & * & \cdots & * & * & \cdots & * \\
		k+2 & - & - & - & \cdots & - & * & \cdots  & * \\
		\vdots &  & \vdots & &  & \vdots & \vdots &  & \vdots  \\
		2k+2 & - & - & - & \cdots & - & * & \cdots & *
	\end{array}
\end{equation}
where, a ``$-$" denotes a zero entry and a ``$*$" denotes a (possibly) non-zero entry. 

\begin{remark}\label{rmk:DifferentFramework}
	We should point out that although $\Sigma_k(C,L_d)$ is  a higher dimensional variety, we are not in the setting of \cite{EinLazAsymptotic}. In that paper, they fix a higher dimensional $X$ and study the syzygies of large degree embeddings of $X$. Here, we consider large degree embeddings of a curve $C$ and study the resulting higher dimensional variety $\Sigma_k(C,L_d)$. In particular, the secant varieties $\Sigma_k(C,L_d)$ are not isomorphic for different values of $d$.
\end{remark}


\subsection{Boij-S\"{o}derberg Decompositions}

Boij-S\"{o}derberg theory provides a novel approach to the study of minimal free resolutions. We will not need the deepest aspects of the theory, so we only describe the bare minimum. We direct the interested reader to \cite{EisenbudSchreyer} for proofs of the core results and \cite{EisenbudErman} for some of the deeper investigations of this phenomenon. We also recommend the excellent survey article \cite{Floystad}. 

As above, let $S = \mathbb{k}[x_0,\dots,x_r]$. For any finitely generated, Cohen-Macaulay graded $S$-module $M$ of dimension 0, we have an associated Betti diagram $\beta(M)$. We consider the data of $\beta(M)$ as an element in the infinite dimensional vector space $\VV = \bigoplus_{i=0}^n\bigoplus_{j \in \ZZ} \QQ$. An element of $\VV$ is called a \emph{formal diagram}.

The Betti diagrams of such modules generate a convex cone in $\VV$. The fundamental results of Boij-S\"{o}derberg identify this cone as the cone spanned by the \emph{pure diagrams}. A pure diagram has a unique non-zero entry in each column. Each pure diagram is characterized (up to scalar multiple of the entire diagram) by its \emph{degree sequence} $\ee = (e_0,\dots,e_n) \in \ZZ^{\oplus n+1}$ with $e_0 < e_1 < \dots < e_n$. Indeed, this follows from the Herzog-K\"{u}hl equations \cite[Section 1.3]{Floystad}. For this paper, we normalize he Betti numbers of the pure diagram $\beta(\ee)$ associated to $\ee$ as
\begin{equation}\label{eq:pureDiagramEntry}
	\kappa_{p,q}(\beta(\ee)) = \begin{cases}
		n!\prod_{i\not= e_p} \frac{1}{|e_i - e_p|} & p+q = e_p \\
		0 & otherwise
	\end{cases}.
\end{equation}
The product in the denominator arises from the Herzog-K\"{u}hl equations, and the numerator is chosen so that this diagram has multiplicity 1. This follows from the computation of the multiplicity of an algebra with pure resolution \cite{HunekeMiller} and \cite[Lemma 1.2]{Erman}. 

More explicitly, for any finitely generated, Cohen-Macaulay, graded $S$ module $M$, we may write $\beta(M)$ as a linear combination $\sum_\ee c_{\ee} \beta(\ee)$ where $c_{\ee} \geq 0$. This decomposition is not unique. However, there is a fan structure on the cone of Betti diagrams, and by fixing a simplex in the fan, one obtains a unique linear combination of pure diagrams corresponding to the extremal points of the simplex. As we shall see in the proof of the main theorem, we need not take this deeper structure into account.


\subsection{Hilbert Numerator}
Let $S = \mathbb{k}[x_0,\dots,x_r]$. For a finitely generated graded $S$-module $M$, the Hilbert series $\mathrm{HS}_M(t) := \sum_{j \in \ZZ} \dim_k(M_j)t^j$ can be uniquely expressed as
\[
	\mathrm{HS}_M(t) = \frac{\HN_M(t)}{(1-t)^{\dim(M)}}
\]
where $\HN_M(t) \in \QQ[t]$ is a polynomial called the \emph{Hilbert numerator}. The \emph{multiplicity} of $M$ is $\HN_M(1)$. Note that when $\dim{M} = 0$, the Hilbert series and Hilbert numerator are equal. 

The dimensions of the graded pieces of $M$ (and thus the Hilbert numerator) are determined by the Betti diagram \cite[Chapter 1]{GeoOfSyzygy}. Thus, we can and do extend the notion of Hilbert numerator and multiplicity to formal diagrams.

\section{Asymptotic Purity}\label{sec:Purity}

In this section, we study the behavior of Boij-S\"{o}derberg decompositions of the Betti diagram of $\Sigma_k(C,L_d)$ as $d \to \infty$. That is, we consider $C \subseteq \PP H^0(C,L_d)$ embedded by the complete linear system and let $\Sigma_k(C,L_d)$ be the secant variety of $k$-planes. We always assume that $d \geq 2g+2k+1$ so that Theorem \ref{thm:ENP} applies.

Recall that the Betti table of $\Sigma_k(C,L_d)$ has the form given in \eqref{eq:secantBettiTable}. Apart from the row of weight $(k+1)$, the possible non-zero entries are confined to a $(k+1) \times g$ box. In particular, the number of possible pure diagrams in any Boij-S\"{o}derberg decomposition of the Betti table is a bounded function of $d$. Let us fix notation for the possible pure diagrams that can appear in the Boij-S\"{o}derberg decomposition of $\beta(\Sigma_k(C,L_d))$. We denote the set $\{0,\dots,n\}$ by $[n]$. Fix a genus $g\geq 1$, so that $r_d = d-g$. Let $\ii = (i_0,\dots,i_k)$ be a tuple of integers with $0\leq i_0 \leq i_1 \leq \dots \leq i_k$. We define $\pi_k(\ii;d)$ to be the pure diagram of formal multiplicity 1 associated to the degree sequence
\[
	\ee_k(\ii ;d) := [r_d+1] \setminus \{1,2,\dots,k+1,r_d+1 - (i_{k} +k), \dots, r_d+1 - (i_1+1), r_d+1 - i_0\}.
\]

\begin{example}\label{ex:pureDiagramExamples}
	Let $k =1$, $g = 3$, and $d=11$. So $r_{11} = 8$. Then, 
	\[
	    \ee_1((1,3);11) = [9] \setminus \{1,2,5,8\} = \{0,3,4,6,7,9\},
	\]
	and $\pi_1((1,3);11)$ is of the form
	\[
		\begin{array}{c|cccccc}
			& 0 & 1 & 2 & 3 & 4 & 5 \\ \hline
			0 & * & - & - & - & - & - \\
	 		1 & - & - & - & - & - & - \\
			2 & - & * & * & - & - & - \\
			3 & - & - & - & * & * & - \\
			4 & - & - & - & - & - & *
		\end{array}.
	\]
	Let $k=2$, $g =4$, and $d = 15$. So $r_d = 11$. Then 
	\[
	    \ee_2((1,3,3);15) = [12] \setminus \{1,2,3,7,8,11\} = \{0,4,5,6,8,10,12\},
	\] 
	and $\pi_2((1,3,3);15)$ is of the form
	\[
		\begin{array}{c|ccccccc}
			& 0 & 1 & 2 & 3 & 4 & 5 & 6  \\ \hline
			0 & * & - & - & - & - & - & - \\
			1 & - & - & - & - & - & - & - \\
			2 & - & - & - & - & - & - & - \\
			3 & - & * & * & * & - & - & - \\
			4 & - & - & - & - & - & - & - \\
			5 & - & - & - & - & * & * & - \\
			6 & - & - & - & - & - & - & * \\ 
		\end{array}.
	\] 
	To motivate this choice of notation, note that a pure diagram in a Boij-S\"{o}derberg decomposition of $\Sigma_k(C,L_d)$ has a nontrivial row in degree $k+1$ and then ``jumps" down at most $k+1$ times. The indices $i_j$ indicate the columns where the jumps occur (counting from the last column of the diagram). The number of jumps is the number of positive entries. So for instance, the diagram $\pi_1((1,3), 11)$ ``jumps" down twice, once at the last column and once at the third to last column.
	Zeros in the vector $\ee_{k}(\ii ;d)$ indicate that a jump does not occur. Indeed, the diagram $\pi_k((0,\dots,0);d)$ is the pure diagram with a single non-zero row of weight $k+1$ (other than the $\kappa_{0,0}$ entry which is always non-zero). That is, this diagram does not ``jump down" at all.
	\hfill \qed
\end{example}

In order to compare the Betti diagrams for $\Sigma_k(C,L_d)$ as $d \to \infty$, it is natural to define
\[
	\overline{\beta}(\Sigma_k(C,L_d)) := \frac{1}{\deg(\Sigma_k(C,L_d))} \beta(\Sigma_{k}(C,L_d)))
\]
so that $\overline{\beta}(\Sigma_k(C,L_d))$ has multiplicity 1 for any $d$. As a consequence, we have the following boundedness property of the Boij-S\"{o}derberg coefficients.

\begin{lemma}\label{lem:BSDecomp}
	Let $L_d$ be a line bundle on $C$ of degree $d \geq 2g+2k+1$. For any Boij-S\"{o}derberg decomposition
	\[
		\overline{\beta}(\Sigma_k(C,L_d)) = \sum_{0 \leq i_0 \leq \cdots \leq i_k \leq g} c_{(i_0,\dots,i_k);d}\pi_{k}((i_0,\dots,i_k);d),
	\]
	we have $\sum_{\ii} c_{\ii;d} = 1$ and $0 \leq c_{\ii;d}\leq 1$.
\end{lemma}
\begin{proof}
	The fact that we need only consider diagrams $\ii = (i_0,\dots,i_k)$ with each $i_j \leq g$ follows from the discussion above. The nonnegativity of the coefficients $c_{\ii;d}$ is a feature of the Boij-S\"{o}derberg decomposition. Finally, since multiplicity is linear, we see that
	\[
		\sum_\ii c_{\ii;d} = \sum_\ii c_{\ii;d}\mult(\pi_k(\ii;d)) = \mult(\overline{\beta}(\Sigma_k(C,L_d))) = 1.
	\]
	Since each $c_{\ii;d} \geq 0$, the expression above implies $c_{\ii;d} \leq 1$.
\end{proof}

To prove the main theorem, we mimic the strategy of \cite{Erman}. In particular, we compare the coefficient of the highest degree term of the Hilbert numerator of $\overline{\beta}(\Sigma_k(C,L_d))$ and of its Boij-S\"{o}derberg decomposition.

\begin{lemma}\label{lem:hilbertNum}
	Let $L_d$ be a line bundle on $C$ of degree $d \geq 2g+2k+1$. The Hilbert numerator of $\overline{\beta}(\Sigma_k(C,L_d))$ satisfies
	\[
		\HN_{\overline{\beta}(\Sigma_k(C,L_d))}(t) = \left(\frac{{g+k \choose k+1}}{\deg(\Sigma_k(C,L_d))}\right) t^{2k+2} + (\text{lower order terms in }t).
	\]
	If $\ii = (i_0,\dots,i_k)$ with $1 \leq i_0 \leq \cdots \leq i_k \leq g$, then the Hilbert numerator of $\pi_k(\ii; d)$ satisfies
	\[
		\HN_{\pi_{k}(\ii;d)}(t) = \left(\frac{\prod_{j=0}^k (i_j + j)}{(r_d+1)\prod_{i=k+1}^{2k}(r_d-i)} \right)t^{2k+2} + (\text{lower order terms in } t).
	\]
\end{lemma}
\begin{proof}
	In general, the Hilbert numerator of a module $M$ is a polynomial of degree $\reg(M) -1$. By taking general hyperplane sections, we may assume $M$ is a module of finite length. Then the leading coefficient of the Hilbert series is the bottom-right entry of the Betti table. 
	
	Now, we address the specific cases at hand. Recall from Theorem \ref{thm:ENP}(\ref{thm:ENP-part:regularity}) that $\reg(\Sigma_k(L_d)) = 2k+3$, so the Hilbert numerator of $\beta(\Sigma_k(C,L_d))$ has degree $2k+2$. The bottom-right Betti number in the minimal free-resolution of $\Sigma_k(L_d)$ is ${g+k \choose k+1}$ by Theorem \ref{thm:ENP}(\ref{thm:ENP-part:FinalBettiNumber}). Since the Hilbert numerator is linear with respect to scaling of Betti tables, we see that the bottom right Betti number of $\overline{\beta}(\Sigma_k(L_d))$ is ${g+k \choose k+1}/ \deg(\Sigma_k(L_d))$ as required.
	
	Since $i_j > 0$ for all $j$, we see that $\pi_k(\ii ;d)$ has Castelnuovo-Mumford regularity $2k+3$. As above, it suffices to compute the bottom right Betti number $\kappa_{r_d-2k-1,2k+2}$ which is given by the formula \eqref{eq:pureDiagramEntry}. Since each $i_j \geq 1$, the largest degree in the degree sequence is $r_d + 1$. Thus, the bottom-right Betti number is given by
	\begin{align*}
	    \kappa_{r_d-2k-1, 2k+2} &= (r_d-2k-1)!\left(\prod_{\substack{0 \leq \ell \leq r_d \\ \ell \notin \{1,2,\dots,k+1, r_{d}+1-(i_k+k),\dots,r_d + 1 - i_0\}}}(r_d + 1 - \ell)\right)^{-1} \\
	    &= \frac{\prod_{j=0}^k (i_j + j)}{(r_d + 1)\prod_{\ell=k+2}^{2k+1} (r_d+1-\ell)} \\
	    &= \frac{\prod_{j=0}^k (i_j + j)}{(r_d+1)\prod_{i=k+1}^{2k}(r_d-i)}
	\end{align*}
	as claimed.
\end{proof}

\begin{lemma}\label{lem:degreeOfSecant}
	The degree of $\Sigma_k(C,L_d)$ satisfies
	\[
		\deg(\Sigma_k(C,L_d)) = \frac{1}{(k+1)!}r_d^{k+1} + (\text{lower order terms in } r_d).
	\]
\end{lemma}
\begin{proof}
	This is immediate from Theorem \ref{thm:ENP}(\ref{thm:ENP-part:degree}).
\end{proof}

\begin{theorem}\label{thm:Main}
	The normalized Betti table $\overline{\beta}(\Sigma_{k}(C,L_d))$ converges to $\pi_k((g,\dots,g);d)$ as $d \to \infty$. That is, given a Boij-S\"{o}derberg decomposition $\overline{\beta}(\Sigma_{k}(C,L_d)) = \sum_{\ii} c_{\ii;d} \pi_k(\ii;d)$ for each $d$, we have
	\[
		\lim_{d \to \infty} c_{\ii;d} = \begin{cases}
			1 & \ii = (g,\dots,g) \\
			0 & \text{otherwise}
		\end{cases}.
	\]
\end{theorem}

\begin{proof}
	By Lemmas \ref{lem:hilbertNum} and \ref{lem:degreeOfSecant}, we see that the coefficient of $t^{2k+2}$ in the Hilbert numerator of $\overline{\beta}(\Sigma_k(C,L_d))$ is 
	\begin{align*}
		\frac{{g+k \choose k+1}}{\deg(\Sigma_k(C,L_d))} & = {g+k \choose k+1} \cdot \frac{(k+1)!}{r_d^{k+1}} + \delta_d = \frac{\prod_{\ell=0}^k (g+\ell)}{r_d^{k+1}} + \delta_d
	\end{align*}
	where $r_d^{k+1}\delta_d \to 0$ as $d \to \infty$. Similarly, we see that the growth rate of the $t^{2k+2}$ coefficient of the Hilbert numerator of $\pi_k(\ii ;d)$ is 
	\[
		\frac{\prod_{j=0}^k (i_j+j)}{r_d^{k+1}} + \varepsilon_{\ii;d}
	\]
	where $r^{k+1}_d\varepsilon_{\ii;d} \to 0$ as $d \to \infty$. The Hilbert numerator is an additive invariant. Applying the decomposition of Lemma \ref{lem:BSDecomp}, we obtain $\HN_{\overline{\beta}(\Sigma_k(C,L_d))}(t) = \sum c_{\ii;d} \HN_{\pi_k(\ii ;d)}(t)$. Comparing the $t^{2k+2}$ coefficients of each side of this equality yields
	\[
		\frac{\prod_{\ell=0}^k(g+\ell)}{r_d^{k+1}} + \delta_d = \sum_{1 \leq i_0 \leq \cdots \leq i_k \leq g} \frac{\prod_{j=0}^k (i_j+j)}{r_d^{k+1}} + \varepsilon_{\ii;d}.
	\]
	Multiplying by $r_d^{k+1}$ and taking $d \to \infty$ gives
	\[
		\prod_{\ell=0}^k(g+\ell) = \lim_{d\to \infty} \sum_{1 \leq i_0 \leq \cdots \leq i_k  \leq g} c_{\ii;d}\prod_{j=0}^k (i_j+j).
	\] 
	Since the diagrams $\overline{\beta}(\Sigma_k(C,L_d))$ and have multiplicity 1, we see that $\sum c_{\ii ;d} = 1$ and $0 \leq c_{\ii ;d} \leq 1$. On the other hand, note that the product $\prod_{j=0}^k (i_j+j)$ is maximized when $i_0 = \cdots = i_k = g$, and in this case $\prod_{j=0}^k (i_j+j) = \prod_{\ell=0}^l (g+k)$. So $c_{(g,\dots,g);d} \to 1$ while $c_{\ii;d} \to 0$ for $\ii \not= (g,\dots,g)$, as required.
\end{proof}

Note that Theorem \ref{thm:Main} demonstrates that every entry in the bottom row that is not guaranteed to vanish by \cite{EinNiuPark} is nonzero for large $d$. 

\begin{corollary}
	For $d \gg 0$, $\kappa_{p,2k+2}(\Sigma_k(C,L_d)) \not= 0$ if and only if $r_d - 2k - g \leq p \leq r_d - 2k - 1$.
\end{corollary}
\begin{proof}
	The coefficients in a Boij-S\"{o}derberg decomposition are non-negative. So if a pure diagram appears in the decomposition of a free resolution, any nonzero entries in the pure diagram are nonzero in the free resolution. For $d \gg 0$, the main theorem implies that the diagram $\pi_k((g,\dots,g);d)$ appears in a Boij-S\"{o}derberg decomposition of $\overline{\beta}(\Sigma_k(C,L_d))$ with nonzero coefficient. Since the $(p,2k+2)$ Betti number of this pure diagram is nonzero for $p$ in the indicated range, the result follows.
\end{proof}

\section{Normal Distribution of Betti Numbers}\label{sec:NormalDist}

In this section, we study the distribution of the Betti numbers in the minimal free resolution of $\Sigma_k(C,L_d)$. By Theorem \ref{thm:ENP}, we need only consider the growth of the groups $K_{p,k+1}(\Sigma_{k}(C,L_d))$. Using Theorem \ref{thm:Main}, we observe that asymptotically, one only needs to analyze the Betti numbers of the pure diagram $\pi_k((g,\dots,g);d)$ for which there is an explicit formula. 

Our analysis relies on Stirling's approximation \cite[\S3.1]{Durrett}. In particular, for a fixed sequence $p_d$ of integers such that $p_d \to \frac{r_d}{2} + \frac{a\sqrt{r_d}}{2}$ (for some fixed $a \in \mathbb{R}$)\footnote{By $p_d\to \frac{r_d}{2} + \frac{a\sqrt{r_d}}{2}$, we mean that $\lim_{d \to \infty} \frac{2p_d - r_2}{\sqrt{r_d}} = a$}, we have 
\begin{equation}\label{eq:stirling}
	\lim_{d \to \infty} \frac{\sqrt{2\pi r_d}}{2^{r_d+1}} {r_d \choose p_d} = e^{-a^2}.
\end{equation}
We refer to this equality as Stirling's formula. In making asymptotic comparisons, we write $f(d) \sim g(d)$ if $\lim_{d \to \infty} \frac{f(d)}{g(d)} = 1$.

\begin{lemma}\label{lem:BettiFormula}
	For $1 \leq p \leq r_d-g-2k-1$, we have
	\begin{align*}
		\kappa_{p,k+1}(\pi_k((g,\dots,g);d)) &= {r_d-2k \choose p} \cdot \frac{p}{p+k+1} \cdot \frac{\prod_{i=k}^{2k} (r_d-p-g-i)}{(r_d-2k)\prod_{i=k}^{2k-1}(r_d-p-i)}
	\end{align*}
\end{lemma}
\begin{proof}
	The degree sequence of $\pi_k((g,\dots,g);d)$ is
	\[
		\ee_k((g,\dots,g);d) = (0,k+2,k+3,\dots,p+k+1,\dots,r_d-k-g,r_d-g+2,\dots,r+1).
	\]
	So by definition,
	\begin{equation}\label{eq:defInProof}
		\kappa_{p,k+1}((g,\dots,g);d) = \frac{(r_d-2k-1)!}{(p+k+1)(p-1)!(r_d-2k-g-p-1)!\prod_{i=0}^{g-1} (r_d-p-k-i)}.
	\end{equation}
	Then since 
	\begin{align*}
		(r_d-2k-g-p-1)!\prod_{i=0}^{g-1} (r_d-p-k-i) &= \frac{(r_d-p-k)!}{\prod_{i=k}^{2k}(r_d-g-p-i)}
	\end{align*}
	we get
	\begin{align*}
		\kappa_{p,k+1}(\pi_k((g,\dots,g);d)) &= \frac{(r_d-2k-1)!}{(p-1)!(r_d-p-k)!}\cdot \frac{\prod_{i=k}^{2k}(r_d-g-p-i)}{p+k+1}.
	\end{align*}
	To conclude, we notice that
	\begin{align*}
		\frac{(r_d-2k-1)!}{(p-1)!(r_d-p-k)!} &= {r_d-2k-1 \choose p-1}\cdot  \prod_{i=k}^{2k-1}\frac{1}{(r_d-p-i)} \\
		&= {r_d-2k \choose p} \frac{p}{(r_d-2k)}\prod_{i=k}^{2k-1}\frac{1}{(r_d-p-i)}.
	\end{align*}
	The second equality above follows from the relation ${n \choose k} = \frac{n}{k}{n-1 \choose k-1}$.
\end{proof}

\begin{lemma}\label{lem:growthRate}
	For $1 \leq p_d \leq r_d-g-2k-1$, we have
	\[
		\kappa_{p_d,k+1}(\pi_k((g,\dots,g);d)) \sim \frac{1}{2^{2k+1}}{r_d\choose p_d}
	\]
	as $p_d \to \frac{r_d}{2} + \frac{a\sqrt{r_d}}{2}$.
\end{lemma}
\begin{proof}
	We analyze the terms appearing in the statement of Lemma \ref{lem:BettiFormula}. Notice that for $1\leq i \leq k$, we have
	\[
		\lim_{d \to \infty} \frac{r_d-p_d-g-i}{r_d-p_d-i} = 1 ~~~~ \text{ and } ~~~~ \lim_{d \to \infty} \frac{p_d}{p_d+k+1} = 1.
	\]
	Furthermore, 
	\[
		\lim_{d \to \infty} \frac{r_d-p_d-g-2k}{r_d-2k} = \frac{1}{2}.
	\]
	Finally, the equality
	\[
		\lim_{d \to \infty} \frac{{r_d-2k\choose p_d}}{{r_d \choose p_d}} = \frac{1}{2^{2k}}
	\]
	completes the proof.
\end{proof}

\begin{theorem}\label{thm:normalDistribution}
	Let $C$ be a curve of genus $g \geq 1$. Let $L_d$ be a line bundle of degree $d$ on $C$, so $r_d = d-g$. Fix a sequence $\{p_d\}$ of integers such that $p_d \to \frac{r_d}{2} + \frac{a\sqrt{r_d}}{2}$ for some real number $a$. Then 
	\[
		F_{g,k}(d)\kappa_{p_d,k+1}(\Sigma_k(C,L_d)) \to e^{-a^2}
	\]
	where $F_{g,k}(d) = \frac{(k+1)!}{2^{r_d - 2k} r_d^{k}} \cdot \sqrt{\frac{2\pi}{r_d}}$.
\end{theorem}
\begin{proof}
	Consider a Boij-S\"{o}derberg decomposition of the normalized Betti table of $\Sigma_k(C,L_d)$ 
	\[
		\overline{\beta}(\Sigma_k(C,L_d)) = \sum c_{\ii;d} \pi_k(\ii;d).
	\]
	Then we may write the $(p_d,k+1)$ Betti number as
	\[
		\kappa_{p_d,k+1}(\overline{\beta}(\Sigma_k(C,L_d))) = \sum c_{\ii;d} \kappa_{p_d,k+1}(\pi_k(\ii;d)).
	\]
	Notice that that $\kappa_{p_d,k+1}(\pi_k(g,\dots,g);d)$ grows at least as fast each of the $\kappa_{p_d,k+1}(\pi_k(\ii;d))$. To see this, one performs a similar calculation as in Lemmas \ref{lem:BettiFormula} and \ref{lem:growthRate}, but this is less involved as one needs only find the growth rate rather than the coefficient\footnote{Indeed, analyzing the proof of Lemma \ref{lem:BettiFormula}, one sees that for other diagrams, the formula in Equation \eqref{eq:defInProof} will be the same except for the product in the denominator. So one obtains a similar growth rate for other diagrams.}. So as $d \to \infty$, Theorem \ref{thm:Main} implies that
	\[
		\kappa_{p_d,k+1}(\overline{\beta}(\Sigma_k(C,L_d))) \sim \kappa_{p_d,k+1}(\pi_d((g,\dots,g);d))
	\]
	since $\lim c_{\ii ;d} = 0$ for $\ii \not= (g,\dots,g)$. Recall that $\overline{\beta}(\Sigma_k(C,L_d)) = \frac{1}{\deg(\Sigma_k(C,L_d))} \beta(\Sigma_k(C,L_d))$. By Lemma \ref{lem:degreeOfSecant}, we have $\deg(\Sigma_k(C,L_d)) = \frac{1}{(k+1)!}r_d^{k+1} + O(r_d^k)$. Thus
	\[
		\kappa_{p_d,k+1}(\Sigma_k(C),L_d) \sim \frac{r_d^{k+1}}{(k+1)!}\kappa_{p_d,k+1}(\pi_d((g,\dots,g);d)) \sim \frac{r_d^{k+1}}{(k+1)!2^{2k+1}}{r_d \choose p_d}.
	\]
	Then the statement follows from Stirling's formula \eqref{eq:stirling}.
\end{proof}

\begin{remark}\label{rmk:NoExplicitFormula}
	Normal distribution of the Betti numbers of curves (i.e. $k=0$ above) was proven in \cite{EinErmanLaz} using an explicit formula for the Betti numbers. Our proof does not require an explicit formula, and in fact, the author is not aware of the existence of an explicit formula for the syzygies of $\Sigma_k(C,L_d)$ in the weight $(k+1)$ row. Thus, our argument shows that normal distribution of Betti numbers for curves follows from the asymptotic purity statement in \cite{Erman}.
\end{remark}

\bibliographystyle{alpha}
\bibliography{AsymptoticSyzygiesSecants}{}

\end{document}